\documentclass{amsart}
\usepackage{graphicx}
\usepackage{amsfonts}
\newtheorem{tm}{Theorem}

\newtheorem{defi}{Definition}

\newtheorem{rem}{Remark}

\newtheorem{ex}{Example}

\newtheorem{prop}{Proposition}
\newtheorem{nota}{Notation}

\newtheorem{prob}{Problem}

\begin{document}
\title{Moduli of roots of hyperbolic polynomials and Descartes' rule of signs}
\author{Vladimir Petrov Kostov}
\address{Universit\'e C\^ote d’Azur, CNRS, LJAD, France} 
\email{vladimir.kostov@unice.fr}

\begin{abstract}
  A real univariate polynomial with all roots real is called hyperbolic. By 
  Descartes' rule of signs for hyperbolic polynomials (HPs) with all
  coefficients nonvanishing, a HP with $c$ sign changes and $p$ sign
  preservations in the sequence of its coefficients has exactly $c$ positive
  and $p$ negative roots. For $c=2$ and for degree $6$ HPs, we
  discuss the question: When the
  moduli of the $6$ roots of a HP are arranged in the increasing order on the
  real
  half-line, at which positions can be the moduli of its two positive roots
  depending on the positions of the two sign changes in the sequence of
  coefficients?\\ 

  {\bf Key words:} real polynomial in one variable; hyperbolic polynomial; sign
  pattern; Descartes' rule of signs\\ 

{\bf AMS classification:} 26C10; 30C15

\end{abstract}

\maketitle

\section{Introduction}

The classical Descartes' rule of signs states that given a real univariate
degree $d$ polynomial $P$, the number $pos$ of its positive roots is not
larger than the number $c$ of the sign changes in the sequence of its
coefficients. In the present paper we consider only polynomials with all
coefficients nonvanishing and with positive leading coefficients. Thus the
number of sign changes for the polynomial $P(-x)$ is equal to the number $p$
of sign preservations for the polynomial $P(x)$ hence one has $neg\leq p$,
where $neg$ stands for the number of negative roots of $P$. If $P$ is
{\em hyperbolic}, i.e. with all roots real, then the conditions

$$c~+~p~=~d~=~pos~+~neg~~~\, ,~~~\, pos~\leq ~c~~~\, {\rm and}~~~\,
neg~\leq ~p~,$$
imply $pos=c$ and $neg=p$. We consider only the {\em generic case} when all
moduli of roots are distinct. Suppose that the $d$ moduli of roots are
arranged in the increasing order on the positive half-line. Then one can
formulate the following problem:

\begin{prob}\label{prob1}
  Knowing the positions of the sign changes and sign
  preservations in the sequence of coefficients of a given hyperbolic
  polynomial (HP), what positions can occupy
  the moduli of its positive roots in this arrangement?
  \end{prob}

For $c=1$, the exhaustive answer to this problem (for any degree $d$) is given
in~\cite{Ko1}. Also in~\cite{Ko1} one can find the answer to
Problem~\ref{prob1}
for $d\leq 5$ and $c=2$. In the present paper we give the answer to it
for $d=6$ and $c=2$.

\begin{rem}
  {\rm For $d\leq 5$, it is sufficient to study Problem~\ref{prob1}
for $c\leq 2$, because the polynomial $(-1)^dP(-x)$ has $p$
sign changes and $c$ sign preservations in the sequence of its coefficients.
Thus if $P(x)$ has more than $2$ sign changes, then $(-1)^dP(-x)$ has not more
than $2$ sign changes and one can consider $(-1)^dP(-x)$ instead of $P(x)$. 
For $d\geq 6$ this is not so. In
particular, for $d=6$, HPs with $c=3$ remain such when $P(x)$ is replaced by
$(-1)^dP(-x)$.}
\end{rem}

The paper is structured as follows. In Section~\ref{secknown} we recall
some definitions and results of \cite{Ko1} and we mention some problems
related to Problem~\ref{prob1}. In Section~\ref{secsol} we first remind which 
cases have to be considered for $d=6$ according to the signs of the coefficients
of the HP, and then we resolve Problem~\ref{prob1} in each of these cases.

\section{Definitions and known results\protect\label{secknown}}

\begin{defi}
  {\rm A {\em sign pattern (SP)} of length $d+1$ is a sequence of $d+1$
    signs $+$ or~$-$. We say that the polynomial $P:=x^d+\sum _{j=0}^{d-1}a_jx^j$
    defines the SP $(+$, sgn$(a_{d-1})$, $\ldots$, sgn$(a_0))$. We consider
    polynomials with positive leading coefficients, so the SPs we deal with
    begin with~$+$. In the proofs we use also SPs some of whose components
    are zeros.}
\end{defi}

\begin{nota}
  {\rm We denote the moduli of the two positive roots of a HP (with two sign
    changes in the sequence of its coefficients) by $0<\beta <\alpha$. By
    $0<\gamma _1<\cdots <\gamma _{d-2}$ we denote the moduli of its negative
    roots. By $(a,b,w)$, $a$, $b$, $w\in \mathbb{N}\cup 0$, $a+b+w=d-2$ we
    denote the case when $\gamma _a<\beta <\gamma _{a+1}$ and
    $\gamma _{a+b}<\alpha <\gamma _{a+b+1}$ setting $\gamma _0:=0$ and
    $\gamma _{d-1}:=\infty$. By $\Sigma _{m,n,q}$, for $c=2$ (resp. by
    $\Sigma _{m,n}$, for $c=1$) we denote the SP consisting
    of $m$ pluses followed by $n$ minuses followed by $q$ pluses, where
  $m+n+q=d+1$ (resp.  consisting
    of $m$ pluses followed by $n$ minuses, where $m+n=d+1$). For $c=1$,
    the moduli of the positive and negative roots are denoted by $\alpha$ and
    $\gamma _1<\cdots <\gamma _{d-1}$; the case $(a,b)$, $a+b=d-1$, means
    $\gamma _a<\alpha <\gamma _{a+1}$. The notation $(a,b,w)$ and $(a,b)$ is
  different from the one used in~\cite{Ko1}.}
  \end{nota}

\begin{defi}
  {\rm (1) Given a SP one defines its corresponding
    {\em canonical arrangement} as
    follows. The increasing order of moduli of positive and negative roots
    on the real half-line coincides with the order of sign changes and sign
    preservations respectively when the SP is read backward. Example: for
    $d=6$ and $c=2$, the SP $\Sigma _{2,4,1}=(+,+,-,-,-,-,+)$ when read backward
    gives the following order of sign changes and sign preservations:
    $(c,p,p,p,c,p)$. Hence the 
    canonical arrangement defined by this SP is 
    $\beta <\gamma _1<\gamma _2<\gamma _3<\alpha <\gamma _4$ which is the case
    $(0,3,1)$. More generally, for the SP $\Sigma _{m,n,q}$, the corresponding
    canonical arrangement defines the case $(q-1,n-1,m-1)$.

    (2) We say that a SP $\Sigma _{m,n,q}$ (resp. $\Sigma _{m,n}$)
    is {\em realizable} in the case
    $(a,b,w)$ (resp. $(a,b)$) if there exists a degree $d$ HP with $c=2$ and
    $p=d-2$ (resp. with $c=1$ and $p=d-1$) defining
    this SP, with distinct moduli of its roots which define the case $(a,b,w)$
    (resp. the case $(a,b)$). We can also say that the case $(a,b,w)$ is
    realizable with the SP $\Sigma _{m,n,q}$.}
\end{defi}

The following result can be found in~\cite{Ko1}:

\begin{tm}\label{tmknown}
  (1) Every SP has a canonical realization, i.e.
  is realizable in the case defined by its canonical arrangement. The
  SPs $\Sigma _{1,d}$, $\Sigma _{d,1}$, $\Sigma _{1,d-1,1}$ and $\Sigma _{m,1,d-m}$
  are realizable only in this case.

  (2) The SPs $\Sigma _{d-2,2,1}$ and $\Sigma _{d-3,3,1}$ are realizable 
  in the case $(1,0,d-3)$. They are realizable only in this case and in
  cases of the kind $(0,b,w)$, $b+w=d-2$.
  For $n\geq 4$, the SP
  $\Sigma _{m,n,1}$ is realizable only in cases of the kind $(0,b,w)$, $b+w=d-2$.
\end{tm}

\begin{rem}\label{remrevert}
  {\rm Given a degree $d$ polynomial $P(x)$ we define its
    {\em reverted} polynomial $P^R$ as $P^R(x):=x^dP(1/x)$. It is clear that if
    the SP $\Sigma _{m,n,q}$ (resp. $\Sigma _{m,n}$)
    is realizable in the case $(a,b,w)$ (resp. $(a,b)$) by the
    HP $P(x)$, then the SP $\Sigma _{q,n,m}$ (resp. $\Sigma _{n,m}$)
    is realizable in the case $(w,b,a)$ (resp. $(b,a)$) 
    by the polynomial sgn$(P(0))P^R(x)$.}
  \end{rem}

Another problem about real univariate polynomials (not necessarily hyperbolic)
inspired by Descartes' rule of signs is the following one:

\begin{prob}\label{prob2}
  Suppose that a SP with $c$ sign changes and $p$ sign preservations is given.
  For which pairs
  of nonzero integers $(pos, neg)$ satisfying the conditions $pos\leq c$,
  $neg\leq p$ and $c-pos\in 2\mathbb{N}\cup 0\ni p-neg$ do there exist such
  polynomials defining the given SP and having exactly $pos$ positive and
  $neg$ negative roots, all distinct?
\end{prob} 

The problem seems to have been formulated for the first time in \cite{AJS}.
Its exhaustive answer for $d\leq 8$ is to be found in  \cite{Gr},
\cite{AlFu}, \cite{FoKoSh} and \cite{KoCzMJ} (and this answer is
not trivial). An interesting particular case for $d=11$ is considered in
\cite{KoMB}. In \cite{FoNoSh} a tropical analog of Descartes' rule of signs is
formulated. Different aspects of the theory of HPs are exposed in \cite{Ko}.
For metric inequalities involving moduli of roots of polynomials see
\cite{AKNR} and~\cite{Fo}.

\section{Resolution of Problem~\protect\ref{prob1} for $d=6$
  \protect\label{secsol}}

For $d=6$, the SPs with $c=2$ which need to be considered
are the following ones:
$\Sigma _{1,5,1}$, $\Sigma _{m,1,q}$ ($m+q=6$), $\Sigma _{k,6-k,1}$
($k=2$, $3$ and $4$), $\Sigma _{2,3,2}$ and $\Sigma _{3,2,2}$. The remaining
SPs are obtained from these ones by reversion, see Remark~\ref{remrevert}.
We consider different SPs in the subsequent subsections. In
Subsection~\ref{subsecsum} we summarize the results.

In the proofs we use the following notation:

\begin{nota}\label{notaproofs}
  {\rm We remind that we consider HPs having four negative and two positive
    roots. These roots are denoted by}

  $$\xi _1<\xi _2<\xi _3<\xi _4<0<\xi _5<\xi _6~.$$
  {\rm For
  the roots of the derivatives we use the notation}

  $$\zeta _1<\zeta _2<\zeta _3<\zeta _4<\zeta _5~~~,~~~\xi _j<\zeta _j<\xi _{j+1}$$
  {\rm (the latter inequalities result from Rolle's theorem). One has
    $\zeta _3<0<\zeta _4$, if the last but first sign of the SP is $+$, and
    $\zeta _4<0<\zeta _5$,
  if it is~$-$.}
  \end{nota}

\begin{rem}
  {\rm In the proofs we use one-parameter deformations of given HPs in which
    the deformation parameter is considered only for these nonnegative values
    for which the deformed polynomial is hyperbolic.}
  \end{rem}

  \begin{rem}
    {\rm For $n$ fixed, one could ask the question what the possible values
      of the quantity $b$ can be. As we shall see, for $d=6$ and $n=2$, one
      has $b\leq 4$. The following example (with $d=7$ and $n=2$) shows
      that $b$ can attain the value $5$:}
   
   $$\begin{array}{ccl}
      D&:=&(x-0.9)(x+0.98)(x+0.99)(x+1)(x+1.01)(x+1.02)(x-1.1)\\ \\ &=&
      x^7+3x^6+0.9895x^5-5.0505x^4-5.09899496x^3\\ \\ &&+0.90101496x^2+
      2.94951496x+0.9895050396~.
     \end{array}$$
    {\rm The polynomial $D$ defines the SP $\Sigma _{3,2,3}$ and realizes the
      case $(0,5,0)$. To obtain such examples with arbitrarily large values of
      $d\geq 7$ (and
      with $n=2$) it suffices to multiply $D$ by polynomials of the form
      $\prod _{\mu =1}^{\mu ^*}(1+\varepsilon _{\mu}x)
      \prod _{\nu =\mu ^*+1}^{\mu ^*+\nu ^*}(x+\varepsilon _{\nu})$, where
      $\varepsilon _i$ are small positive quantities. Indeed, for
      $1\leq \mu \leq \mu ^*$, the moduli of the negative roots
      $-1/\varepsilon _{\mu}$ are large whereas for
      $\mu ^*+1\leq \mu \leq \mu ^*+\nu ^*$, the moduli of the negative roots
      $-\varepsilon _{\mu}$ are small; the polynomial then defines the SP
      $\Sigma _{3+\mu ^*,2,3+\nu ^*}$ and realizes the case $(\nu ^*, 5, \mu ^*)$.} 
  \end{rem}

\subsection{The SPs $\Sigma _{1,5,1}$ and $\Sigma _{m,1,q}$, $m+q=6$}

These two cases have only canonical realizations, see Theorem~\ref{tmknown}.
This means
that for the HPs realizing the SP $\Sigma _{1,5,1}$ one has
$\beta <\gamma _1<\gamma _2<\gamma _3<\gamma _4<\alpha ~(*)$. This is the case
$(0,4,0)$. 

\begin{ex}\label{ex1}
  {\rm Consider the polynomial $P^*:=(x-0.01)(x+0.25)^4(x-1)$, i.e.}

  $$\begin{array}{lll}
    P^*&=&
    x^6-0.01x^5-0.625x^4-0.30625x^3\\ \\ &&-0.05546875x^2
    -0.0033203125x+0.0000390625~.\end{array}$$
  {\rm This polynomial defines the SP $\Sigma _{1,5,1}$. Hence
    for $\varepsilon >0$ small enough, the polynomial}

  $$(x-0.01)(x+0.25-2\varepsilon )(x+0.25-\varepsilon )(x+0.25+\varepsilon )
  (x+0.25+2\varepsilon )(x-1)$$
  {\rm also defines the SP $\Sigma _{1,5,1}$ and
    has six distinct real roots whose moduli satisfy conditions~$(*)$.}
\end{ex}

\begin{ex}\label{ex2}

  {\rm For $m=1$, $2$ and $3$, the following HPs define
    the corresponding SPs $\Sigma _{m,1,q}$:} 

  $$\begin{array}{lllc}
    (x+0.01)^4(x-1)^2&=&x^6-1.96x^5+0.9206x^4+0.038804x^3&\\ \\
    &&+0.00059201x^2+0.00000398x+10^{-8}&,\\ \\
    (x+0.01)^3(x-1)^2(x+4)&=&x^6+2.03x^5-6.9397x^4+3.790601x^3&\\ \\
    &&+0.117902x^2+0.001193x+0.000004&{\rm and}\\ \\
    (x+0.01)^2(x-1)^2(x+4)^2&=&x^6+6.02x^5+1.1201x^4-23.9794x^3\\ \\
    &&+15.5201x^2+0.3176x+0.0016&.\end{array}$$
  {\rm The reverted of these HPs
    define these SPs with $m=5$, $4$ and $3$ respectively. One can define
    one-parameter deformations of these HPs in which the multiple
    roots split into simple real ones while preserving the signs of the roots
    and of the coefficients of the corresponding polynomial (as this is done
    in Example~\ref{ex1}). For small nonzero values of the deformation
    parameter, the deformations are canonical realizations of the corresponding
  SPs.}
\end{ex}

\subsection{The SP $\Sigma _{2,4,1}$}

For any HP realizing this SP, one has $\beta <\gamma _1$,
see part (2) of Theorem~\ref{tmknown}. Hence a priori the realizable
cases are of the form $(0,b,4-b)$, $0\leq b\leq 4$. The ones with $b=2$,
$3$ and $4$ are realizable: 

$$\begin{array}{l}
  (0,2,2)~~~:~~~
  (x-0.001)(x+0.3)(x+0.4)(x-1)(x+1.01)(x+1.02)=\\ \\
  x^6+1.729x^5-0.16053x^4-1.6063012x^3\\ \\
  -0.83950954x^2-0.122782884x+
  0.000123624~,\\ \\
  (0,3,1)~~~:~~~
  (x-0.001)(x+0.3)(x+0.4)(x+1)(x-1.01)(x+1.02)=\\ \\
  x^6+1.709x^5-0.19491x^4-1.6229468x^3\\ \\ -0.84194086x^2-0.122780436x+
  0.000123624~,\\ \\
  (0,4,0)~~~:~~~
  (x-0.001)(x+0.3)(x+0.4)(x+1)(x+1.01)(x-1.02)=\\ \\
  x^6+1.689x^5-0.22889x^4-1.6393128x^3\\ \\
  -0.84432446x^2-0.122778036x+0.000123624~.
\end{array}$$

\begin{prop}
  The case $(0,0,4)$ is not realizable with the SAP $\Sigma _{2,4,1}$.
\end{prop}

\begin{proof}
  Indeed, suppose that the polynomial $P$ realizes the SAP $\Sigma _{2,4,1}$ in
  the case $(0,0,4)$. Hence $\zeta _4<0<\zeta _5$ and 
  $-\xi _4=\gamma _1>\alpha =\xi _6$, see Notation~\ref{notaproofs}.
  This means that

  \begin{equation}\label{eqcontra}
    -\zeta _1>-\zeta _2>-\zeta _3>-\xi _4>\xi _6>\zeta _5~.
  \end{equation}
  The HP $P'$ defines the SP $\Sigma _{2,4}$. It follows from
  \cite[Corollary~1]{Ko1}
  that for $d=5$, the cases $(3,1)$ and $(2,2)$ are realizable with the SAP
  $\Sigma _{2,4}$, but the cases $(1,3)$ and $(0,4)$ are not. This is a
  contradiction with~(\ref{eqcontra}).
  \end{proof}

\begin{prop}
  The case $(0,1,3)$ is not realizable with the SP $\Sigma _{2,4,1}$.
  \end{prop}

\begin{proof}
  Suppose that the HP $P$ realizes the SP $\Sigma _{2,4,1}$ in the case
  $(0,1,3)$.
We use Notation~\ref{notaproofs}. Consider for $t\geq 0$
the one-parameter deformation $P_t:=P+tx^4(x-\xi _6)$. 
As $t$ increases, the root
$\xi _6$ does not change, $\xi _1$, $\xi _3$ and $\xi _5$ decrease while
$\xi _2$ and $\xi _4$ increase; the SP does not change. For some value $t_0>0$
of $t$, at least one of the two things happens:
\vspace{1mm}

A) the roots $\xi _2$ and $\xi _3$ coalesce;

B) one has $|\xi _4|=|\xi _5|$, i.e. $-\xi _4=\xi _5$.
\vspace{1mm}

Set $Q:=P_{t_0}$. If A) takes place, then we consider the one-parameter
deformation $Q_s:=Q-s(x-\xi _2)^2x^2$, $s\geq 0$. As $s$ increases, the
double root
$\xi _2=\xi _3$ does not change, $\xi _1$ and $\xi _5$ decrease while
$\xi _4$ and $\xi _6$ increase; the SP does not change. Then for some
$s=s_0>0$, either B) or C) takes place, with
\vspace{1mm}

C) one has $|\xi _6|=|\xi _2|$, i.e. $-\xi _2=\xi _6$.
\vspace{1mm}

We denote by AB) ``A) followed by B)''. Thus if A) takes place, then there
exists a HP defining the SP $\Sigma _{2,4,1}$ for which either
AB) or AC) takes place.

Suppose that B) takes place. Consider the one-parameter deformation
$Q_u:=Q+uQ^*$, $Q^*:=(x^2-\xi _4^2)(x-\xi _6)$, $u\geq 0$. The roots
$\xi _4$, $\xi _5$ and
$\xi _6$ do not change, $\xi _1$ and $\xi _3$ decrease while $\xi _2$ increases.
One has

$$Q^*=x^3-\xi _6x^2-(\xi _4)^2x+\xi _4^2\xi _6=x^3-U_2x^2-U_1x+U_0~~~,~~~U_j>0~.$$
Hence as $u$ increases, the signs of the last three coefficients of $Q_u$
do not change and the first three coefficients do not change at all.
The coefficient of
$x^3$ cannot become negative. Indeed, $Q_u$ is hyperbolic and four sign
changes in the sequence of its coefficients means four positive roots which
is not the case. Thus $u$ can be increased only until for some value $u_0$,
A) takes place, so one can assume that either AB) (or BA) which is the same)
or AC) takes place.

Set $R:=Q_{u_0}$. Suppose that AB) takes place.
Consider the one-parameter deformation
$R_v:=R-vR^*$, $R^*:=(x-\xi _2)^2(x^2-\xi _4^2)$, $v\geq 0$. One has

$$R^*=x^4-2\xi _2x^3+(\xi _2^2-\xi _4^2)x^2+2\xi _2\xi _4^2x-\xi _2^2\xi _4^2=
x^4+R_3x^3+R_2x^2-R_1x-R_0~~~,~~~R_j>0~.$$
Thus in the deformation $R_v$ only the sign of the linear term can change.
The roots $\xi _i$, $2\leq i\leq 5$, do not change, $\xi _1$ decreases and
$\xi _6$ increases, so for some value $v_0>0$ either C) or D) takes place,
with
\vspace{1mm}

D) the coefficient of $x$ equals $0$.
\vspace{1mm}

Hence if AB) takes place, then it suffices to consider the possibility
ABC) or ABD) to take place.

Suppose that AC) takes place.
Consider the deformation $R_r:=R+rR^{\dagger}$,
$r\geq 0$,

$$R^{\dagger}:=(x+\xi _6)^2(x-\xi _6)=x^3+\xi _6x^2-\xi _6^2x-\xi _6^3~.$$
The roots $\xi _2=\xi _3$ and $\xi _6$ do not change, $\xi _1$ and $\xi _5$
decrease while $\xi _4$ increases. In the deformation $R_r$ the coefficient of
$x^3$ or the one of $x^2$ or both of them cannot become positive,
because this would mean at least
three sign changes, i.e. at least three positive roots which is impossible.
The constant term of $R_r$ cannot vanish. Indeed, in this case after rescaling
one can set $\xi _6=1$ hence $R_r=xT$, where

$$T=(x+1)^2(x-1)(x+g)(x+h)~~~,~~~g=-\xi _1>1~~~,~~~h=-\xi _4\in (0,1)~.$$
Observe that the SP $\Sigma _{2,4,1}$ implies that $R_r'(0)<0$ hence it is
$\xi _5$ and not $\xi _4$ that vanishes. The coefficient of $x^3$ in $T$
equals $-1+g+h+gh>0$ whereas it
must be negative. However, as $R^{\dagger}(0)<0$, for
$r=-R(0)/R^{\dagger}(0)>0$, one has $R_r(0)=0$, i.e. for $r=r_0$, the constant
term does vanish. This contradiction shows that
one cannot have AC). So one cannot have ABC) either and only the possibility
ABD) remains. 

We rescale the variable $x$ so that $\xi _4=-1=-\xi _5$. We set
$\xi _2=\xi _3=-g$, $g>1$, $\xi _6=A>1$, $\xi _1=-B$, $B>1$. So we consider the
polynomial

$$F:=(x+g)^2(x^2-1)(x+B)(x-A)=x^6+F_5x^5+\cdots +F_0~.$$
As $F_1=g(-gB+gA+2AB)$, the condition $F_1=0$ yields $g=g_*:=2AB/(B-A)$.
Hence $B>A$, because $g>0$. 
One gets

$$F_4|_{g=g_*}=(-2A^2B^2-B^2+2AB-A^2+3AB^3+3A^3B)/(-B+A)^2~.$$
However the inequalities $B>A>1$ imply $2AB^3>2A^2B^2$, $AB^3>B^2$ and
$AB>A^2$, i.e. $F_4|_{g=g_*}>0$ which is in contradiction with the SP
$\Sigma _{2,4,1}$. This contradiction proves the proposition.

\end{proof}

\subsection{The SP $\Sigma _{3,3,1}$}

According to part (2) of Theorem~\ref{tmknown}, the HPs realizing this SP
either satisfy the inequalities
$\gamma _1<\beta <\alpha <\gamma _2<\gamma _3<\gamma _4$,
which is the case $(1,0,3)$ realizable by the HP

$$\begin{array}{l}
  (x+0.98)(x-0.99)(x-1)(x+2.05)(x+2.1)(x+40)=\\ \\
  
  x^6+43.14x^5+124.7533x^4-41.23068x^3\\ \\
  -294.614531x^2-0.116529x+167.06844~,
\end{array}$$
or they realize one of the cases $(0,b,4-b)$, $0\leq b\leq 4$. These cases
are also realizable: 

$$\begin{array}{l}
  (0,0,4)~~~:~~~
  (x-0.1)(x-9)(x+9.6)(x+9.7)(x+9.8)(x+9.9)=\\ \\
  x^6+29.9x^5+216.35x^4-1448.135x^3\\ \\ -24185.4276x^2-78877.71684x+
  8131.05216~,\\ \\ 
  (0,1,3)~~~:~~~
  (x-0.1)(x+0.99)(x-1)(x+1.01)(x+1.02)(x+40)=\\ \\
  x^6+41.92x^5+76.6179x^4-9.305992x^3-81.6975778x^2-32.6139222x+4.079592~,\\ \\
  (0,2,2)~~~:~~~ 
  (x-0.1)(x+0.99)(x+0.995)(x-1)(x+1.02)(x+40)=\\ \\
  x^6+41.905x^5+76.00425x^4-9.835474x^3\\ \\ -81.0232111x^2-32.0695689x+
  4.019004~,\\ \\
  (0,3,1)~~~:~~~
  (x-0.1)(x+0.99)(x+0.995)(x+0.999)(x-1)(x+40)=\\ \\
  x^6+41.884x^5+75.145665x^4-10.55580655x^3\\ \\ -80.08192694x^2-31.3281913x+
  3.9362598~,\\ \\
  (0,4,0)~~~:~~~
  (x-0.1)(x+9.6)(x+9.7)(x+9.8)(x+9.9)(x-10)=\\ \\
  x^6+28.9x^5+177.45x^4-2014.585x^3\\ \\ -27835.3426x^2-87541.52424x+9034.5024~.
\end{array}$$

\subsection{The SP $\Sigma _{4,2,1}$}

For the HPs realizing this SP there exist two possibilities, see
Theorem~\ref{tmknown}.
The first of them is to satisfy the inequalities
$\gamma _1<\beta <\alpha <\gamma _2<\gamma _3<\gamma _4$, and this 
is the case $(1,0,3)$ realizable by the polynomial

$$\begin{array}{l}
  (x+1)(x-1.5)(x-1.6)(x+10)(x+11)(x+12)=\\ \\
  
  x^6+30.9x^5+292x^4+539.1x^3-2946.2x^2-55.2x+3168~.
\end{array}$$
The second is one of the cases $(0,b,4-b)$, $0\leq b\leq 4$,
to be realizable. For $b=0$, $1$ and $2$ we provide examples of HPs
realizing the corresponding cases:

$$\begin{array}{l}
  (0,0,4)~~~:~~~(x-1)(x-4)(x+5)(x+6)(x+100)(x+101)=\\ \\
  x^6+207x^5+11285x^4+56273x^3-233286x^2-1046480x+1212000~,\\ \\
  (0,1,3)~~~:~~~(x-1)(x+2)(x-4)(x+5)(x+100)(x+101)=\\ \\
  x^6+203x^5+10481x^4+15957x^3-216482x^2-214160x+404000~,\\ \\
  (0,2,2)~~~:~~~(x-1)(x+2.1)(x+3)(x-4)(x+1000)(x+1001)=\\ \\
  x^6+2001.1x^5+1.0011849\times
  10^6x^4+69673.7x^3-1.52373859\times 10^7x^2\\ \\
  -1.10606748\times 10^7x+
  2.52252\times 10^7~.
  \end{array}$$

\begin{prop}
  The case $(0,4,0)$ is not realizable with the SP $\Sigma _{4,2,1}$.
\end{prop}

\begin{proof}
  Suppose that the SP $\Sigma _{4,2,1}$ is realizable in the case $(0,4,0)$ by
  the HP $P$. Then the SP $\Sigma _{1,2,4}$ is realizable in the case $(0,4,0)$
  by the HP $Q:=P^R$. For the roots of $Q$ and $Q'$ we use
  Notation~\ref{notaproofs}. 
  The HP $Q'$ defines the SP $\Sigma _{1,2,3}$ and as
  $|\xi _{\nu}|\in (\xi _5,\xi _6)$, $1\leq \nu \leq 4$, one has
  $|\zeta _k|>\zeta _4$, $k=1$, $2$ and $3$. One has $\zeta _5>\zeta _4$;
  the position of $\zeta _5$ w.r.t. $|\zeta _k|$, $1\leq k\leq 3$,
  cannot be specified. 

  The HP $(Q')^R$ defines the SP $\Sigma _{3,2,1}$ and has roots
  $\eta _j=1/\zeta _j$ for which one has
  $|\eta _k|<\eta _4$ and $\eta _k<0$, $k=1$, $2$ and $3$. Notice that $\eta _4$
  is the largest root $\eta _j$. Indeed, there are two positive roots
  $\eta _4$ and $\eta _5$, and as $\zeta _5>\zeta _4$, this implies
  $\eta _5<\eta _4$. 
  From part (2) of Theorem~\ref{tmknown} and from \cite[Proposition~1]{Ko1}
  (which claims that for $d=5$, the case $(0,3,0)$ is not realizable with
  the SP $\Sigma _{3,2,1}$) one deduces that this is impossible.
  \end{proof}

\begin{prop}
  The case $(0,3,1)$ is not realizable with the SP $\Sigma _{4,2,1}$.
\end{prop}

\begin{proof}
  Suppose that the HP $P$ realizes the SP $\Sigma _{4,2,1}$ in the
  case $(0,3,1)$. We use Notation~\ref{notaproofs}. Consider the polynomial
  $V:=x^3(x-\xi _2)(x-\xi _3)(x-\xi _4)$. It defines the SP $(+,+,+,+,0,0,0)$.
  Hence throughout the deformation $P_t:=P+tV$, $t\geq 0$, the SP does not
  change, the roots $\xi _2$, $\xi _3$ and $\xi _4$ remain the same, $\xi _1$
  and $\xi _5$ increase while $\xi _6$ decreases. Hence for some $t=t_0>0$,
  at least one of the following things happens:
  \vspace{1mm}

  A) one has $|\xi _4|=|\xi _5|$, i.e. $-\xi _4=\xi _5$;

  B) one has $|\xi _1|=|\xi _6|$, i.e. $-\xi _1=\xi _6$;

  C) one has $|\xi _2|=|\xi _6|$, i.e. $-\xi _2=\xi _6$.
  \vspace{1mm}

  Suppose that A) takes place. Set $Q:=P_{t_0}$ and consider the deformation
  $Q_s:=Q+sQ^{\triangle}$, $Q^{\triangle}:=(x^2-\xi _5^2)(x^2-\xi _6^2)$, $s\geq 0$.
  The SP defined by the polynomial $Q^{\triangle}$ is $(0,0,+,0,-,0,+)$, so the
  SP does not change throughout the deformation $Q_s$. The roots $\xi _4$,
  $\xi _5$ and $\xi _6$ do not change; $\xi _1$ increases without reaching
  $-\xi _6$; $\xi _2$ increases and $\xi _3$ decreases. Hence for some
  $s=s_0>0$, one has
  \vspace{1mm}

  D) the roots $\xi _2$ and $\xi _3$ coalesce.
  \vspace{1mm} 

  By rescaling the $x$-axis one obtains the condition $\xi _2=\xi _3=-1$. Hence
  one can set $-\xi _4=\xi _5=g\in (0,1)$, $\xi _6=A>1$ and $-\xi _1=B>A$. The
  corresponding polynomial equals $K:=(x+1)^2(x^2-g^2)(x+B)(x-A)$.
  Its coefficient of $x^3$ is

  $$-2g^2+g^2A-A-Bg^2+B-2AB=-2g^2-(1-g^2)A-Bg^2+(1-A)B-AB<0$$
  which contradicts the SP $\Sigma _{4,2,1}$.

  Suppose that B) takes place. Consider the deformation $Q_v:=Q+vQ^{\dagger}$,
  $Q^{\dagger}:=(x^2-\xi _6^2)x^2$, $v\geq 0$. The polynomial $Q^{\dagger}$ defines
  the SP $(0,0,+,0,-,0,0)$, so the SP does not change throughout the
  deformation $Q_v$. The roots $\xi _1$ and $\xi _6$ do not change. The root
  $\xi _4$ increases while $\xi _5$ decreases both keeping away from~$0$
  (because $Q_v(0)$ does not change). As $\xi _2$ increases and $\xi _3$
  decreases, for some $v=v_0>0$, one has either A)
  (a case already considered) or D).

  If B) and D) take place, then by rescaling the $x$-axis so that
  $\xi _2=\xi _3=-1$, one sets
  $-\xi _1=\xi _6=A>1$ and $-\xi _4=g\in (0,1)\ni h=\xi _5$. The corresponding
  polynomial equals $(x+1)^2(x^2-A^2)(x-h)(x+g)$, with coefficient of $x^3$
  equal to

  $$-2A^2-A^2g+g+hA^2-h-2gh=(h-1)A^2+(g-A^2)-A^2g-h-2gh<0$$
  which contradicts the SP $\Sigma _{4,2,1}$.

  Suppose that C) takes place. Consider the family of polynomials
  $Q+\delta Q^{\bullet}$, where $\delta \geq 0$ and
  $Q^{\bullet}:=(x^2-\xi _6^2)(x^2-\xi _5^2)$. The
  polynomial $Q^{\bullet}$ defines the SP $(0,0,+,0,-,0,+)$, so the SP does not
  change throughout the deformation. The roots $\xi _2$, $\xi _5$ and $\xi _6$
  do not change, $\xi _1$ and
  $\xi _4$ increase ($\xi _4$ never becomes equal to $-\xi _5$) while $\xi _3$
  decreases. For some $\delta =\delta _0>0$, takes place either D) or

  E) the roots $\xi _1$ and $\xi _2$ coalesce.

  If E) takes place, then after rescaling one obtains for
  $\xi _1=-1=\xi _2=-\xi _6$ the polynomial

  $$G:=(x+1)^2(x-1)(x+g)(x+h)(x-k)~~~,~~~k=\xi _5~~~,~~~g=-\xi _3~~~,~~~
  h=-\xi _4~,$$
  hence $0<k<h<g<1$. The coefficient of $x^3$ equals

  $$-1-g-h+gh+k-kg-kh-ghk=-1-g(1-h)-(h-k)-kg-kh-ghk<0$$
  which contradicts the SP $\Sigma _{4,2,1}$. If D) takes place, then we rescale
  the $x$-axis to obtain the polynomial $G$ with the equalities
  $\xi _2=-1=\xi _3=-\xi _6$, $k=\xi _5$,
  $g=-\xi _1$, $h=-\xi _4$, hence $0<k<h<1<g$. With the same form of the
  coefficient of $x^3$ one concludes that the negative sign of this coefficient
  contradicts the SP $\Sigma _{4,2,1}$.

  The proposition is proved.
  \end{proof}
  
  \subsection{The SP $\Sigma _{2,3,2}$}

  \begin{prop}\label{prop232}
    All cases $(a,b,w)$, $a+b+w=4$, $0\leq a$, $b$, $w$ $\leq 4$, are realizable
    with the SP $\Sigma _{2,3,2}$.
  \end{prop}

  \begin{proof}

    The SP $\Sigma _{2,3,2}$ is center-symmetric, this is why if this SP
    is realizable in the case $(a,b,w)$ by a HP $P$, then it is realizable in
    the case $(w,b,a)$ by the HP $P^R$. We prove the proposition by
    exhibiting examples
    of HPs which realize $\Sigma _{2,3,2}$ in the indicated cases. The above
    observations allow to skip some of the cases.

  $$\begin{array}{l}
    (1,0,3)~~~:~~~(x+0.01)(x-0.1)(x-1)(x+1.01)(x+1.02)(x+1.03)=\\ \\
      x^6+1.97x^5-0.1253x^4-2.067553x^3\\ \\
      -0.87576764x^2+0.097559534x+0.001061106~,\\ \\
    (1,1,2)~~~:~~~(x+0.01)(x-0.1)(x+0.99)(x-1)(x+1.02)(x+1.03)=\\ \\
      x^6+1.95x^5-0.1445x^4-2.045655x^3\\ \\
      -0.85653356x^2+0.095648466x+0.001040094~,\\ \\ 
    (1,2,1)~~~:~~~(x+0.01)(x-0.1)(x+0.98)(x+0.99)(x-1)(x+1.03)=\\ \\
      x^6+1.91x^5-0.1817x^4-2.001931x^3\\ \\
      -0.81930584x^2+0.091937534x+0.000999306~,\\ \\
    (1,3,0)~~~:~~~(x+0.01)(x-0.1)(x+0.97)(x+0.98)(x+0.99)(x-1)=\\ \\
      x^6+1.85x^5-0.2345x^4-1.936645x^3\\ \\
      -0.76643456x^2+0.086638466x+0.000941094~,\end{array}$$
    $$\begin{array}{l}
    (0,2,2)~~~:~~~
    (x-1)(x+1.1)(x+2)(x-2.1)(x+2.2)(x+2.3)=\\ \\
    x^6+4.5x^5-0.25x^4-24.205x^3-23.6436x^2+19.2214x+23.3772~,\\ \\
    (0,1,3)~~~:~~~(x-1)(x+1.1)(x-2)(x+2.05)(x+2.1)(x+2.15)=\\ \\
    x^6+4.4x^5-0.0425x^4-21.8665x^3-20.921675x^2+17.068025x+20.36265~,\\ \\
    (0,4,0)~~~:~~~(x-1)(x+2.9)(x+3)(x+3.1)(x+3.2)(x-8)=\\ \\
    x^6+3.2x^5-46.01x^4-291.172x^3-487.418x^2+129.968x+690.432~,\\ \\
    (2,0,2)~~~:~~~(x+0.8)(x+0.9)(x-1)(x-5)(x+5.1)(x+5.2)=\\ \\
    x^6+6x^5-22.25x^4-156x^3-72.1556x^2+147.9336x+95.472~,\\ \\
    (0,0,4)~~~:~~~(x-1)(x-1.001)(x+1.002)(x+1.01)(x+1.02)(x+1.1)=\\ \\
    x^6+2.131x^5-0.867672x^4-4.26624106x^3\\ \\
    -1.268949846x^2+2.135240980x+1.136621926~.
  \end{array}$$
\end{proof}

  \subsection{The SP $\Sigma _{3,2,2}$}

  \begin{prop}
    The cases $(a,b,w)$ with $w\geq 1$ are realizable with the SP
    $\Sigma _{3,2,2}$.
    \end{prop}

  \begin{proof}
    For $d=5$ and for the SP $\Sigma _{2,2,2}$, all cases
  $(a,b,w)$ with $a+b+w=3$ are realizable, see~\cite{Ko1}.
  Hence for $d=6$, all corresponding cases $(a,b,w+1)$ are also realizable.
  Indeed, if the degree $5$ HP $P(x)$ defines the SP $\Sigma _{2,2,2}$ and
  realizes the case $(a,b,w)$, then for
  $\varepsilon >0$ small enough, the HP $(1+\varepsilon x)P(x)$ defines
  the SP $\Sigma _{3,2,2}$ and realizes the
  case $(a,b,w+1)$. Indeed, the root $-1/\varepsilon$ has the largest of
  the moduli of its roots.
  \end{proof}

  Thus it remains to consider the cases $(a,b,0)$ with
  $a+b=4$. The case $(0,4,0)$ is realizable by the HP

  $$\begin{array}{l}
    (x-1)(x+1.9)(x+1.91)(x+1.92)(x+2)(x-2.1)=\\ \\
    
    x^6+4.63x^5+0.5412x^4-24.36394x^3-28.469668x^2+17.398152x+29.264256~.
  \end{array}$$

  \begin{prop}
    The case $(4,0,0)$ is not realizable with the SP $\Sigma _{3,2,2}$.
  \end{prop}

  \begin{proof}
    Suppose that the HP $P$ realizes the case $(4,0,0)$ with
    the SP $\Sigma _{3,2,2}$. For the roots of $P$ and $P'$ we use
    Notation~\ref{notaproofs}; in particular, $\xi _4<0<\xi _5$ and
    $\zeta _3<0<\zeta _4$. The polynomial $P'$ defines the SP $\Sigma _{3,2,1}$
    and for its roots
    one has $\zeta _5>|\zeta _k|$, $k=1$, $2$ and $3$. It follows from
    part (2) of Theorem~\ref{tmknown} and \cite[Proposition~1]{Ko1} (by which
    for $d=5$, the case $(0,3,0)$ is not realizable with
    the SP $\Sigma _{3,2,1}$) 
    that this is impossible.
    \end{proof}

  \begin{prop}
    The cases $(1,3,0)$, $(2,2,0)$ and $(3,1,0)$ are not realizable with
    the SP $\Sigma _{3,2,2}$.
  \end{prop}

  \begin{proof}
    Suppose that the HP $P$ realizes the case $(1,3,0)$ or
    $(2,2,0)$ or $(3,1,0)$ with
    the SP $\Sigma _{3,2,2}$. We use Notation~\ref{notaproofs}. Consider the
    one-parameter deformation

    $$P_t:=P+t(x-\xi _1)(x-\xi _6)x^2~~~,~~~t\geq 0~.$$
    The polynomial
    $(x-\xi _1)(x-\xi _6)x^2$ defines the SP $(0,0,+,-,-,0,0)$, because
    $|\xi _1|<\xi _6$; therefore the SP does not change throughout the
    deformation $P_t$. The roots $\xi _1$ and $\xi _6$ do not change,
    $\xi _2$ and $\xi _4$ increase while $\xi _3$ and $\xi _5$ decrease. Hence
    there exists $t_0>0$ such that for $P_{t_0}$ at least one of the two things
    holds true:
    \vspace{1mm}

    \begin{tabbing}
      A) \= in case $(1,3,0)$, one has $|\xi _4|=|\xi _5|$,
      i.e. $-\xi _4=\xi _5$;\\ 
      \> in case $(2,2,0)$, one has $|\xi _3|=|\xi _5|$,
      i.e. $-\xi _3=\xi _5$;\\ 
      \> \hspace{4mm}(the possibility $|\xi _2|=|\xi _5|$ also exists,
      but we consider it\\
      \> \hspace{4mm}only in
      the case $(3,1,0)$, because in both cases $(2,2,0)$ and $(3,1,0)$\\
      \> \hspace{4mm}it gives 
      $\xi _2=-\xi _5$, $-\xi _1<\xi _6$);\\ 
      \> in case $(3,1,0)$, one has
      $|\xi _2|=|\xi _5|$, i.e. $-\xi _2=\xi _5$;\\
      B) \> the roots $\xi _2$ and $\xi _3$ coalesce.
    \end{tabbing}
    \vspace{1mm}

    Consider the case $(1,3,0)$. Suppose that A) takes place.
    Set $Q:=P_{t_0}$. Then throughout
    the deformation

    \begin{equation}\label{eqQs}
      Q_s:=Q+s(x^2-\xi _1^2)(x^2-\xi _5^2)~~~,~~~s\geq 0~,
      \end{equation}
    the SP does not change.
    Indeed, the polynomial $(x^2-\xi _1^2)(x^2-\xi _5^2)$ defines the SP
    $(0,0,+,0,-,0,+)$. The roots $\xi _1$, $\xi _4$ and $\xi _5$ do not change
    while $\xi _6$ decreases (but
    never becomes equal to $-\xi _1$). The root $\xi _2$ increases while
    $\xi _3$ decreases.

    Therefore
    there exists $s_0>0$ for which both A) and B) take place. In this case
    after rescaling of $x$ one can have

    $$\xi _2=\xi _3=-1~~~,~~~-\xi _4=\xi _5=g\in (0,1)~~~,~~~-\xi _1=B>1~~~
    {\rm and}~~~\xi _6=A>B~.$$
    The
    corresponding polynomial equals $(x+B)(x+1)^2(x^2-g^2)(x-A)$ and its
    coefficient of $x^4$ is $2(B-A)+(1-AB)-g^2<0$ which contradicts the SP
    $\Sigma _{3,2,2}$. Hence A) does not take place.

    Suppose that B) takes place. We use again the rescaling of $x$ leading to
    $\xi _2=\xi _3=-1$. Consider the polynomial $S:=(x+1)^2(x-1/2)$.
    It defines the SP $(0,0,0,+,+,0,-)$. Then throughout the deformation

    \begin{equation}\label{eqQv}
      Q_v:=Q-vS~~~,~~~v\geq 0~,
      \end{equation}
    the SP does not change. The roots $\xi _2$ and 
    $\xi _3$ do not change while
    $\xi _1$ and $\xi _6$ increase; $\xi _6$ and does not coalesce with
    $\xi _5$. The root $\xi _5$
    increases, if $\xi _5<1/2$,  decreases, if $\xi _5>1/2$,
    and remains fixed, if $\xi _5=1/2$. 
    As A) does not take place,
    for some $v_0>0$,
    \vspace{1mm}

    C) the root $\xi _1$ coalesces with $\xi _2=\xi _3$.
    \vspace{1mm}

    In this case $Q_{v_0}=(x+1)^3(x+g)(x-h)(x-A)$, where $-\xi _4=g<h=\xi _5<1$
    and $A=\xi _6>1$. The coefficient of $x^4$ in $Q_{v_0}$ equals

    $$3+3g-3h-gh-(3+g-h)A\leq 3+3g-3h-gh-3-g+h=2(g-h)-gh<0$$
    which contradicts the SP $\Sigma _{3,2,2}$. Thus the case $(1,3,0)$ is not
    realizable with the SP $\Sigma _{3,2,2}$.

    Consider the case $(2,2,0)$. Suppose that A) takes place.
    Throughout the deformation $Q_s$ (see (\ref{eqQs})) the
    roots $\xi _1$, $\xi _3$ and $\xi _5$ do not change; 
    $\xi _6$ decreases without becoming equal to $-\xi _1$; $\xi _4$ 
    decreases and $\xi _2$ increases. Hence, for some $s=s_0>0$, there are two
    possibilities. The first is to have A) and B), i.e. one can set 

    $$\xi _2=\xi _3=-1=-\xi _5~~~,~~~-\xi _4=g\in (0,1)~~~,~~~-\xi _1=B>1~~~
    {\rm and}~~~\xi _6=A>B~.$$
    The corresponding polynomial is $W:=(x+1)^2(x-1)(x+g)(x+B)(x-A)$ whose
    coefficient of $x^4$ equals

    $$W_4:=-1+g+B+gB-A-Ag-AB=(B-A)(g+1)+(g-AB)-1<0$$
    which contradicts the SP $\Sigma _{3,2,2}$. The second is to have A) and
    \vspace{1mm}

    D) the root $\xi _4$ coalesces with $\xi _3$.
    \vspace{1mm}

    In this case one sets

    $$\xi _3=\xi _4=-1=-\xi _5~~~,~~~-\xi _2=g>1~~~,~~~-\xi _1=B>g~~~
    {\rm and}~~~\xi _6=A>B~.$$
    The corresponding polynomial equals $W$ and its coefficient of $x^4$ is
    $W_4<0$. So suppose that B) takes place. Then A) takes place as well, and
    this possibility was already rejected.

    Consider the case $(3,1,0)$. Suppose that A) takes place. Throughout
    the deformation $Q_s$ (see (\ref{eqQs})) the roots $\xi _1$, $\xi _2$
    and $\xi _5$ do not
    change, $\xi _6$ decreases, but remains larger than $|\xi _1|$, $\xi _3$
    increases and $\xi _4$ decreases. Hence for some $s=s_0>0$, one has D).
    Consider the polynomial 
    $(x^2-1)(x+g)^2(x+B)(x-A)$, 
    where 
    
    $$-\xi _2=\xi _5=1~~~,~~~-\xi _3=-\xi _4=g\in (0,1)~~~,~~~-\xi _1=B>1~~~{\rm and}~~~\xi _6=A>B~.$$
    The coefficient of $x^4$ equals 
    
    $$-1+g^2+2gB-2Ag-BA=-(1-g^2)-2g(A-B)-BA<0$$
    which contradicts the SP $\Sigma _{3,2,2}$.

    If B) (but not A)) takes place,
    then in the deformation $Q_v$ (see (\ref{eqQv}))
    the roots $\xi _2=\xi _3$ do not change,
    $\xi _1$ and $\xi _6$ increase while $\xi _4$ decreases. We admit that
    $\xi _5$ might increase or decrease or remain fixed. Hence
    for some $v=v^*$, either D) takes place or  
    \vspace{1mm}
    
    E) one has $|\xi _5|=|\xi _1|$, i.e. $\xi _5=-\xi _1$. 
    \vspace{1mm}
    
    If B) and D) take place, then one considers the polynomial
    $(x+1)^3(x-g)(x+B)(x-A)$, where 
    
    $$1=-\xi _2=-\xi _3=-\xi _4~~~,~~~1<g=\xi _5<B=-\xi _1<A=\xi _6~.$$ 
    The coefficient of $x^4$ equals 
    
    $$3-3g+3B-gB-3A+Ag-AB=-3(g-1)-3(A-B)-A(B-g)-gB<0$$
    which contradicts the SP $\Sigma _{3,2,2}$.
    
    If B) and E) take place, then one considers the polynomial 
    $(x+1)^2(x^2-B^2)(x+h)(x-A)$ with 
    
    $$-\xi _4=h<1=-\xi _2=-\xi _3<B=\xi _5=-\xi _1<A=\xi _6~.$$
     The coefficient of $x^4$ equals
     
     $$1-B^2+2h-2A-Ah=-(B^2-1)-2(A-h)-Ah<0$$
     which is a contradiction with the SP $\Sigma _{3,2,2}$.

    The proposition is proved.

    \end{proof}

  \subsection{Summarization of the results\protect\label{subsecsum}}

  In this subsection we summarize the results of the section. For each SP of
  the left column we give in the column Y the cases $(a,b,w)$, $a+b+w=4$, which
  are realizable, and in the column N the ones that are not realizable. We
  list only cases which are allowed by Theorem~\ref{tmknown}. The results
  concerning SPs which are not on the list are obtained by reversion, see
  Remark~\ref{remrevert}.

  $$\begin{array}{lccc}\, \, \, \, {\rm SP}&{\rm Y}&&{\rm N}\\ \\
    (1,5,1)&(0,4,0)&&\\ \\ (m,1,q),~m+q=6&(q-1,0,m-1)&&\\ \\ 
    (2,4,1)&(0,2,2),~(0,3,1),~(0,4,0)&&(0,0,4),~(0,1,3)\\ \\
    (3,3,1)&(1,0,3),~(0,0,4),~(0,1,3),&&\\ &(0,2,2),~(0,3,1),~(0,4,0)&&\\ \\ 
    (4,2,1)&(1,0,3),~(0,0,4),~(0,1,3),&&(0,3,1),~(0,4,0)\\ &(0,2,2)&&\\ \\ 
    (2,3,2)&{\rm all~possible~cases}&&\\ \\
    (3,2,2)&{\rm all~other~cases}&&(4,0,0),~(3,1,0),\\  &&&(2,2,0),~(1,3,0)
  \end{array}$$

\end{document}